\documentclass[12pt,reqno,a4paper]{amsart}

\usepackage{graphicx}
\usepackage{color}
\usepackage{transparent}

\usepackage[latin1]{inputenc}  

\usepackage{mathrsfs}
\usepackage{amssymb,amsmath, amsfonts, amsthm}
\usepackage{latexsym}
\usepackage[dvips]{epsfig}

\newtheorem{theorem}{Theorem}[section]
\newtheorem{proposition}[theorem]{Proposition}
\newtheorem{lemma}[theorem]{Lemma}
\newtheorem{corollary}[theorem]{Corollary}
\newtheorem{remark}[theorem]{Remark}



\usepackage{epsfig,color,xcolor}

\newcommand{\ble}{\begin{lemma}}
\newcommand{\ele}{\end{lemma}}
\newcommand{\be}{\begin{equation*}}
\newcommand{\ee}{\end{equation*}}

\newcommand{\bel}{\begin{equation}}
\newcommand{\eel}{\end{equation}}

\newcommand{\C}{\mathbb{C}}
\newcommand{\ep}{\varepsilon}
\newcommand{\fr}{\frac }

\newcommand{\lap}{\Delta}
\newcommand{\N}{\mathbb{N}}
\newcommand{\na}{\nabla}

\newcommand{\R}{\mathbb{R}}

\renewcommand{\to}{\rightarrow}
\newcommand{\To}{\longrightarrow}

\newcommand{\xip}{x_{i,p}}
\newcommand{\xjp}{x_{j,p}}

\newcommand{\mip}{\mu_{i,p}}

\newcommand{\upp}{u_p}

\def\sideremark#1{\ifvmode\leavevmode\fi\vadjust{\vbox to0pt{\vss
 \hbox to 0pt{\hskip\hsize\hskip1em
 \vbox{\hsize2.1cm\tiny\raggedright\pretolerance10000
  \noindent #1\hfill}\hss}\vbox to15pt{\vfil}\vss}}}%

\begin{document}

\numberwithin{equation}{section}
\parindent=0pt
\hfuzz=2pt
\frenchspacing

\title[]{Morse index and sign changing bubble towers for Lane-Emden problems}

\author[]{Francesca De Marchis, Isabella Ianni, Filomena Pacella}

\address{Francesca De Marchis, University of Roma {\em Tor Vergata}, Via della Ricerca
Scientifica 1, 00133 Roma, Italy}
\address{Isabella Ianni, Second University of Napoli, V.le Lincoln 5, 81100 Caserta, Italy}
\address{Filomena Pacella, University of Roma {\em Sapienza}, P.le Aldo Moro 8, 00185 Roma, Italy}



\thanks{Research partially supported by: FIRB project \textquotedblleft{Analysis and Beyond}\textquotedblright, PRIN $201274$FYK7$\_005$ grant and INDAM - GNAMPA}

\begin{abstract} We consider the semilinear Lane-Emden problem
\begin{equation}\label{problemAbstract}\left\{\begin{array}{lr}-\Delta u= |u|^{p-1}u\qquad  \mbox{ in }\Omega\\
u=0\qquad\qquad\qquad\mbox{ on }\partial \Omega
\end{array}\right.\tag{$\mathcal E_p$}
\end{equation}
where $p>1$ and $\Omega$ is a smooth bounded symmetric domain of $\R^2$.
We show that for families $(\upp)$ of sign-changing symmetric solutions of \eqref{problemAbstract} an upper bound on their Morse index implies concentration of the positive and negative part, $u_p^\pm$, at the same point, as $p\to+\infty$. Then an asymptotic analysis of $u_p^+$ and $u_p^-$ shows that the asymptotic profile of $(u_p)$, as $p\to+\infty$, is that of a tower of two different bubbles. \end{abstract}

\maketitle

\section{Introduction}\label{Introduction}
Let $\Omega$ be a smooth bounded domain of $\R^2$ and consider the Lane-Emden problem
\begin{equation}\label{problem}
\left\{\begin{array}{lr}-\Delta u= |u|^{p-1}u\qquad  \mbox{ in }\Omega\\
u=0\qquad\qquad\qquad\mbox{ on }\partial \Omega
\end{array}\right.
\end{equation}
where $p>1$.

\

The aim of the paper is to show, under some symmetry assumption on $\Omega$, a relation between the Morse index of sign-changing symmetric solutions of \eqref{problem} and their asymptotic profile, as $p\to+\infty$.

\

In order to state precisely our result we need to introduce some notations. For a given family $(u_p)$ of sign changing solutions of \eqref{problem} we denote by

\begin{itemize}
\item $u_p^+=\max (0,u_p)$, $u_p^-=-\min(0, u_p)$
\item $\mathcal{N}_p^{\pm}\subset\Omega$ the positive/negative nodal domain of $u_p$, i.e. $\mathcal{N}_p^{\pm}=\{x\in\Omega:\upp(x)\gtrless0\}$
\item $NL_p$ the nodal line of $\upp$, i.e. $NL_p=\{x\in\Omega:\upp(x)=0\}$
\item $x_p^{\pm}$ the maximum/minimum point in $\Omega$ of $u_p$, i.e. $u_p(x_p^{\pm})=\pm\|u_p^{\pm}\|_{\infty}$
\item $\mu_p^{\pm}:=\frac{1}{\sqrt{p|u_p(x_p^{\pm})|^{p-1}}}$
\item $\widetilde{\Omega}_p^{\pm}:=\frac{\Omega-x_p^{\pm}}{\mu_p^{\pm}}=\{x\in\mathbb R^2: x_p^{\pm}+\mu_p^{\pm}x\in \Omega\}$.
\end{itemize}

\

Recalling that the Morse index $m(v)$ of a solution $v$ of a problem of type \eqref{problem} is the number of the negative eigenvalues of the linearized operator at $v$, we state our main result:
\begin{theorem}\label{teoremaIndMor}
Let $\Omega\subset\R^2$ be a simply connected bounded smooth domain containing the origin $O$ and invariant under the action of a cyclic group $G$ of rotations about the origin with order $|G|\geq2$.
Let $(\upp)$ be a family of sign changing $G$-symmetric  solutions of \eqref{problem} with two nodal regions such that
\begin{equation}\label{assumptionEnergyGenerale}
p\int_{\Omega}|\nabla u_p|^2 dx\stackrel{p\to+\infty}\longrightarrow\beta,\quad \textrm{for some $\beta\in\mathbb R$,}
\end{equation}
and
\begin{equation}\label{assumptionMorseu}
m(\upp)<|G|+1.
\end{equation}
Then, assuming w.l.o.g. that $\|u_p\|_{\infty}=\|u_p^+\|_{\infty}$, we have
\begin{itemize}
\item[i)] $|x_p^{\pm}|\rightarrow O \ \mbox{ as }\ p\rightarrow +\infty$,
\item[ii)] $NL_p$ shrinks to the origin as $p\rightarrow +\infty$,
\item[iii)] the rescaled function $v_p^+(x):=p\frac{u_p(x_p^++\mu_p^+ x)-u_p(x_p^+)}{u_p(x_p^+)}$ defined in $\widetilde{\Omega}_p^{+}$ converges (up to a subsequence) to the regular solution $U$ of \begin{equation}
\label{LiouvilleEquation}
\left\{
\begin{array}{lr}
-\Delta U=e^U\quad\mbox{ in }\R^2\\
\int_{\R^2}e^Udx= 8\pi.
\end{array}
\right.
\end{equation}
 with $U(0)=0$, in $C^1_{loc}(\R^2)$,
\item[iv)] the rescaled function $v_p^-(x):=p\frac{u_p(x_p^-+\mu_p^- x)-u_p(x_p^-)}{u_p(x_p^-)}$ defined in $\widetilde{\Omega}_p^{-}$ converges in
$C^1_{loc}(\R^2\setminus\{x_{\infty}\})$ (up to a subsequence) to a singular solution  $V$ of
\begin{equation}
\label{LiouvilleSingularEquation}
\left\{
\begin{array}{lr}
-\Delta V=e^V+ H\delta_{x_\infty}\quad\mbox{ in }\R^2\\
\int_{\R^2}e^Vdx<\infty
\end{array}
\right.
\end{equation}
where $H$ is a negative suitable constant and $\delta_{x_\infty}$ is the Dirac measure centered at ${x_\infty}=-\lim_{p\rightarrow +\infty}\frac{x_p^-}{\mu_p^-}\neq0$,
\item[v)] $\sqrt{p}u_p\rightarrow 0$ in $C^1_{loc}(\R^2\setminus\{0\})$ as $p\rightarrow +\infty$.
\end{itemize}
\end{theorem}

The assertions of the above theorem show that both $u_p^+$ and $u_p^-$ concentrate at the same point which is the origin and, after suitable rescalings, they have the limit profile of a regular and a singular solution of the Liouville equation in the plane. So the limit profile of $u_p$, as $p\to+\infty$, is that of a \emph{tower of two different bubbles}.

\begin{remark}
According to the classification in \cite{PrajapatTarantello}, if $H\notin-4\pi\N$, the solutions of \eqref{LiouvilleSingularEquation} are radial with respect to $x_\infty$, while, if $H\in-4\pi\N$, they can be either radial with respect to $x_\infty$ or invariant under the action of a cyclic group of rotations of order $\tfrac{H}{4\pi}+1$ (which in our case should be at least $|G|$) about $x_\infty$. We refer to Proposition \ref{prop:scalingNegativo} for further details.
\end{remark}

The first results for problem \eqref{problem} about the existence of sign changing solutions whose positive and negative part concentrate at the same point have been obtained in \cite{GrossiGrumiauPacella2} for nodal radial solutions in the ball and in \cite{DeMarchisIanniPacella2} for nodal symmetric solutions similar to those considered in Theorem \ref{teoremaIndMor}. As compared to \cite{DeMarchisIanniPacella2} the main difference is that there a relation between the asymptotic energy $\beta$ (see \eqref{assumptionEnergyGenerale}) of the solutions and the order of the group $G$ was exploited, while here we use the bound \eqref{assumptionMorseu} on the Morse index.

We believe that this connection between the Morse index and the limit profile of the solutions is the real novelty of our result. It shows once again a deep relation between the information obtained by the linearization and the qualitative properties of the solutions.

Our assumption \eqref{assumptionMorseu} also allows to weaken the hypothesis on the order of the symmetry group $G$ which, in \cite{DeMarchisIanniPacella2}, was assumed to be: $|G|\geq 4e$.
On the other side it should be said that, generally, energy conditions are easier to be checked than Morse index bounds. Indeed in \cite{DeMarchisIanniPacella1} solutions satisfying the energy bound stated in \cite{DeMarchisIanniPacella2} have been proved to exist. Another difference with the result in \cite{DeMarchisIanniPacella2} is that here for the asymptotic analysis of $u^-_p$ we are not able to exclude the non radiality of $v^-_p$.

\

Let us observe that the assumptions of Theorem \ref{teoremaIndMor} are reasonable since the $G$-symmetric solutions found recently in \cite{DeMarchisIanniPacella1} in the case $|G|\geq4$, have two nodal regions, satisfy \eqref{assumptionEnergyGenerale} and we conjecture, supported by numerical evidence and asymptotic computations, that their Morse index should be $4$.  Let us recall that for some symmetric sign changing solutions a lower bound on their Morse index can be obtained, as proved in \cite{AftalionPacella}. This shows in particular that the Morse index of  sign changing radial solutions in a ball is at least $4$ and we expect that in the case of least energy radial sign-changing solutions in a ball, their Morse index is exactly $4$, as we are going to prove in a paper in  preparation.

\

The Theorem \ref{teoremaIndMor} will follow from a slightly more general result where the assumption \eqref{assumptionMorseu} is substituted by the condition
\begin{equation}\label{assumptionMorseIndex}
\max\{m(u_p^+), m(u_p^-)\}<|G|.
\end{equation}
Indeed, since the Morse index $m(\upp)$ of a solution $\upp$ of \eqref{problem} is always larger or equal to $m(u_p^\pm)+1$, it is obvious that \eqref{assumptionMorseu} implies \eqref{assumptionMorseIndex}.

\begin{theorem}\label{teoremaIndMorHpDeb}
Let $\Omega\subset\R^2$ a bounded simply connected smooth domain with the origin $O$ and invariant under the action of a cyclic group $G$ of rotations about the origin with $|G|\geq2$.\\
If $(\upp)$ is a family of sign changing solutions of \eqref{problem} with two nodal regions satisfying \eqref{assumptionEnergyGenerale} and \eqref{assumptionMorseIndex} then the assertions $(i)-(v)$ hold.
\end{theorem}

The proof of Theorem \ref{teoremaIndMorHpDeb} (and hence of Theorem \ref{teoremaIndMor}) is based on several results proved in \cite{DeMarchisIanniPacella2}. Let us point out that a crucial initial step is to show that the solutions considered have the property that their nodal line neither touches the boundary of $\Omega$, nor passes through the origin, i.e. for $\upp$ holds:
\begin{equation}\label{hpNonToccaBordoECentro}
NL_p\cap\partial\Omega=\emptyset\quad\textrm{and}\quad O\not\in NL_p.
\end{equation}
Since the solution $\upp$ considered in the above theorem have two nodal regions, $\eqref{hpNonToccaBordoECentro}$ is a consequence of the following general result whose proof is exactly the same as that of \cite[Lemma 4.1]{DeMarchisIanniPacella1} and \cite[Lemma 4.3]{DeMarchisIanniPacella1} (written there for $|G|\geq4$).

\begin{proposition}
\label{prop:nonTocca}
 If $G$ is a cyclic group of rotations about the origin with $|G|\geq2$ then any $G$-symmetric nodal solution $u_p$ of \eqref{problem} such that $\sharp (u_p)\leq |G|$ satisfies \eqref{hpNonToccaBordoECentro}, where $\sharp(u_p)$ is the number of nodal domains of $u_p$.
\end{proposition}

We believe that \eqref{hpNonToccaBordoECentro} is the crucial qualitative property of the solutions which yields the concentration of $u^+_p$ and $u^-_p$ at the same point.

Moreover let us observe that for sign-changing solutions with any number of nodal regions in any $G$-symmetric domain $\Omega$ the condition \eqref{assumptionMorseu} implies the properties in \eqref{hpNonToccaBordoECentro}. Indeed, we know (cfr. \cite{BenciFortunato}) that
$$\sharp (u_p)\leq m(u_p).$$
hence  \eqref{assumptionMorseu} yields
\begin{equation}
\sharp(u_p)\leq m(u_p)\leq |G|,
\end{equation}
so that again by  Proposition \ref{prop:nonTocca} we get \eqref{hpNonToccaBordoECentro}.

\

The outline of the paper is as follows. In Section \ref{section:PreliminaryGeneral} we recall or prove some results in general bounded, not necessarily symmetric domains. In Section \ref{section:Proof} we give the proof of Theorem \ref{teoremaIndMorHpDeb} as consequence of other results concerning the asymptotic analysis of the negative parts $(\upp^-)$ in $G$-symmetric domains.

\section{Preliminary results in general bounded domains}\label{section:PreliminaryGeneral}

In order to prove Theorem \ref{teoremaIndMorHpDeb}, we follow the scheme of the proof of \cite[Theorem 1.2]{DeMarchisIanniPacella2}, showing that all the steps can be re-obtained under the new assumptions of this paper.

\

We start introducing some notations and recalling some results obtained in \cite{DeMarchisIanniPacella2} on the asymptotic behavior of a family $(\upp)$ of solutions of \eqref{problem}, in a general smooth bounded  domain $\Omega$, satisfying the energy condition \eqref{assumptionEnergyGenerale}.

\

Given a family $(u_p)$ of solutions of \eqref{problem}
and assuming that there exist  $n\in\N\setminus\{0\}$ families of points $(\xip)$, $i=1,\ldots,n$  in $\Omega$ such that
\begin{equation}
\label{muVaAZero}
p|\upp(\xip)|^{p-1}\to+\infty\ \mbox{ as }\ p\to+\infty,
\end{equation}
we define the parameters $\mip$ by
\bel\label{mip}
\mip^{-2}=p |\upp(\xip)|^{p-1},\ \mbox{ for all }\ i=1,\ldots,n.
\eel
By \eqref{muVaAZero} it is clear that $\mip\to0$ as $p\to+\infty$ and that
\begin{equation}\label{RemarkMaxCirca1}
 \forall \epsilon>0 \;\: \exists\, p_{i,\epsilon}\ \mbox{ such that }\
 |u_p(\xip)|\geq 1-\epsilon,\ \ \forall p\geq p_{i,\epsilon}.
\end{equation}
Then we define the concentration set
\bel\label{S}
\mathcal{S}=\left\{\lim_{p\to+\infty}\xip,\,i=1,\ldots,n\right\}\subset\bar\Omega
\eel
and the function
\bel\label{RNp}
R_{n,p}(x)=\min_{i=1,\ldots,n} |x-\xip|, \ \forall x\in\Omega.
\eel

Finally we introduce the following properties:
\begin{itemize}
\item[$(\mathcal{P}_1^n)$] For any $i,j\in\{1,\ldots,n\}$, $i\neq j$,
\[
\lim_{p\to+\infty}\fr{|\xip-\xjp|}{\mip}=+\infty.
\]
\item[$(\mathcal{P}_2^n)$] For any $i=1,\ldots,n$,
\[
v_{i,p}(x):=\fr{p}{\upp(\xip)}(\upp(\xip+\mip x)-\upp(\xip))\To U(x)
\]
in $C^1_{loc}(\R^2)$ as $p\to+\infty$, where
\begin{equation}\label{v0}
U(x)=\log\left(\fr1{1+\fr18 |x|^2}\right)^2
\end{equation}
is the solution of $-\lap U=e^{U}$ in $\R^2$, $U\leq 0$, $U(0)=0$ and $\int_{\mathbb{R}^2}e^{U}=8\pi$.
\item[$(\mathcal{P}_3^n)$] There exists $C>0$ such that
\[
p R_{n,p}(x)^2 |\upp(x)|^{p-1}\leq C
\]
for all $p$ sufficiently large and all $x\in \Omega$.
\end{itemize}

\

The following results have been obtained in \cite{DeMarchisIanniPacella2}.

\

\ble\label{lemma:BoundEnergia}
Let $(\upp)$ be a family of solutions to \eqref{problem} satisfying \eqref{assumptionEnergyGenerale}. Then
\begin{itemize}
\item[\emph{$(i)$}] If $\upp$ changes sign, then  $\|u_p^\pm\|_{L^\infty(\Omega)}^{p-1}\geq \lambda_1$ where $\lambda_1:=\lambda_1(\Omega)$ is the first eigenvalue of the operator $-\Delta$ in $H^1_0(\Omega)$. In particular for the points $x_p^{\pm}$, where the maximum and the minimum are achieved, the analogous of \eqref{muVaAZero} and \eqref{RemarkMaxCirca1} hold.

\item[\emph{$(ii)$}] If, for $n\in\N\setminus\{0\}$, the properties $(\mathcal{P}_1^n)$ and $(\mathcal{P}_2^n)$ hold for families $(\xip)_{i=1,\ldots,n}$ of points satisfying \eqref{muVaAZero}, then
\[
p\int_\Omega |\na\upp|^2\,dx\geq8\pi\sum_{i=1}^n \alpha_i^2+o_p(1)\ \mbox{ as }p\rightarrow +\infty,
\]
where  $\alpha_i:=\liminf_{p\to+\infty}|\upp(\xip)|$.
\end{itemize}
\ele

\begin{proof}
See \cite[Lemma 2.1]{DeMarchisIanniPacella2}.
\end{proof}

\begin{proposition}\label{prop:x1N}
Let $(\upp)$ be a family of solutions to \eqref{problem} and assume that \eqref{assumptionEnergyGenerale} holds. Then there exist $k\in\N\setminus\{0\}$ and $k$ families of points $(\xip)$ in $\Omega$  $i=1,\ldots, k$ such that, after passing to a sequence, \eqref{muVaAZero}, $(\mathcal{P}_1^k)$, $(\mathcal{P}_2^k)$, and $(\mathcal{P}_3^k)$ hold. Moreover, given any family of points $x_{k+1,p}$, it is impossible to extract a new sequence from the previous one such that $(\mathcal{P}_1^{k+1})$, $(\mathcal{P}_2^{k+1})$, and $(\mathcal{P}_3^{k+1})$ hold with the sequences $(\xip)$, $i=1,\ldots,k+1$. At last, we have
\begin{equation}\label{pu_va_a_zero}
\sqrt{p}\upp\to 0\quad\textrm{ in $C^1_{loc}(\bar\Omega\setminus\mathcal{S})$ as $p\to+\infty$.}
\end{equation}
\end{proposition}

\begin{proof}
See \cite[Proposition 2.2]{DeMarchisIanniPacella2}.
\end{proof}

\

Proposition \ref{prop:x1N} was inspired by the paper \cite{Druet} where positive solutions of semilinear elliptic problems with critical exponential nonlinearities in $2$-dimension were studied. Its proof is based on an induction argument, namely one first proves that $(\mathcal{P}_1^1)$, $(\mathcal{P}_2^1)$ hold for points $x_{1,p}$ where $\upp$ achieves $\|\upp\|_{\infty}$ (actually $(\mathcal{P}_1^1)$ is trivially verified) and then one shows that if $(\mathcal{P}_1^n)$, $(\mathcal{P}_2^n)$ are satisfied for some $k\in\N\setminus\{0\}$ then either $(\mathcal{P}_3^n)$ holds true or there exists a point $x_{n+1,p}$, such that the $(n+1)$-tuple $x_{1,p},\ldots,x_{n+1,p}$ fulfills $(\mathcal{P}_1^{n+1})$, $(\mathcal{P}_2^{n+1})$. The procedure necessarily stops by virtue of Lemma \ref{lemma:BoundEnergia} and assumption \eqref{assumptionEnergyGenerale}.

\

Moreover one can easily derive the following corollary.
\begin{corollary}\label{cor:nonvedoNL}
Under the assumptions of Proposition \ref{prop:x1N} if the solutions $u_p$ are sign-changing it follows that
$$
\fr{dist(\xip,\partial\Omega)}{\mip}\stackrel{p\to+\infty}{\to}+\infty\qquad\textrm{and}\qquad\fr{dist(\xip,NL_p)}{\mip}\stackrel{p\to+\infty}{\to}+\infty\qquad\textrm{for all $i\in\{1,\ldots,k\}$}
$$
where, as in Section \ref{Introduction},  $NL_p$ denotes the  nodal line of $\upp$.\\
As a consequence, for any $i\in\{1,\ldots,k\}$, letting $\mathcal{N}_{i,p}\subset\Omega$ be the nodal domain of $u_p$ containing $x_{i,p}$ and setting $u_p^i:=u_p\chi_{\mathcal{N}_{i,p}}$ ($\chi_A$ is the characteristic function of the set $A$),
then the scaling of $u_p^i$ around $x_{i,p}$:
\begin{equation}\label{z_{i,p}}
z_{i,p}(x):=\fr{p}{\upp(\xip)}(\upp^i(\xip+\mip x)-\upp(\xip)),
\end{equation}
defined on $\widetilde{\mathcal{N}}_{i,p}:=\frac{\mathcal{N}_{i,p}-x_{i,p}}{\mu_{i,p}}$, converges to $U$ in $C^1_{loc}(\mathbb R^2)$, where $U$ is the function defined in \eqref{v0}.
\end{corollary}
\begin{proof}
See \cite[Corollary 2.4]{DeMarchisIanniPacella2}.
\end{proof}

\

We point out that, since we are assuming without loss of generality that $\|\upp\|_{\infty}=\|\upp^+\|_{\infty}$, we can take $x^+_p$ as the point $x_{1,p}$ so that directly from the proof of Proposition \ref{prop:x1N} we get the following result for the rescaling about $x^+_p$.

\begin{proposition} \label{rem:x^+=x_1}
Let $(\upp)$ be a family of solutions to \eqref{problem} satisfying \eqref{assumptionEnergyGenerale}. Then the rescaled functions
\begin{equation}
v_p^+(x):=\fr{p}{\upp(x_p^+)}(\upp(x_p^++\mu_p^+ x)-\upp(x_p^+))
\end{equation}
defined on $\widetilde{\Omega}_{p}^+$ (see Section \ref{Introduction} for the definition) converge to $U$ in $C^1_{loc}(\mathbb R^2)$, where $U$ is the  function introduced in \eqref{v0}.
\end{proposition}

\
Now we prove a general proposition on the sign of the first eigenvalue of the linearized operators at $\upp^\pm$:
\[
L^{\pm}_p:=-\Delta -p|u_p^{\pm}|^{p-1},
\]
in the space $H^1_0(\mathcal N^{\pm}_p)$, respectively. Let us denote by $\lambda_j^{\pm}$, $j=1,2,\dots$ respectively, their eigenvalues with homogeneous Dirichlet boundary conditions and let $m(u^{\pm}_p)$ be the Morse index of $u_p^{\pm}$ in $\mathcal{N}_p^{\pm}$, namely  $\lambda_j^{\pm} <0$, for $j=1,\dots, m(u_p^{\pm})$ and $\lambda_{m(u_p^{\pm})+1}^{\pm}\geq 0$.
Moreover for a domain $B\subseteq \mathcal{N}_p^{\pm}$ we denote by $\lambda_j^{\pm}(B)$, $j=1,2,\dots$ the  Dirichlet eigenvalues of $L^{\pm}_p$ in $B$.

\begin{proposition}\label{LemmaFunzioneBella} Let $(\upp)$ be a family of solutions to \eqref{problem} satisfying \eqref{assumptionEnergyGenerale} and let $(x_{i,p})\subset\Omega$, $i=1,\dots ,k$ be families of points as in Proposition \ref{prop:x1N}. Then there exists $\bar r>0$ such  that
\begin{eqnarray*}
&&\lambda_1^{+}\left(B_{ \bar{r}\mu_{i,p}}(x_{i,p})\right)< 0\mbox{ for large }p\mbox{, if }(x_{i,p})\subset \mathcal{N}_p^{+}\\
&&
\lambda_1^{-}\left(B_{ \bar{r}\mu_{i,p}}(x_{i,p})\right)< 0\mbox{ for large }p\mbox{, if }(x_{i,p})\subset \mathcal{N}_p^{-}.
\end{eqnarray*}
where $B_{\bar r\mip}(\xip)$ are the balls centered in $\xip$ of radius $\bar r\mip$.
\end{proposition}
\begin{proof} Without loss of generality, by \eqref{muVaAZero}, we may assume that either $(x_{i,p})\subset\mathcal{N}_p^{+}$ or $(x_{i,p})\subset\mathcal{N}_p^{-}$, for $p$ large. We give the proof in the case $(x_{i,p})\subset \mathcal{N}_p^{+}$, the other case being similar.

Let us consider the linear operators
$$\widetilde{ L_{i,p}^+}:=-\Delta -\frac{|u_p^+(\mu_{i,p} x+ x_{i,p})|^{p-1}}{|u_p( x_{i,p})|^{p-1}}$$
in the space $H^1_0(\widetilde{\mathcal{N}}_{i,p}^+)$ where $\widetilde{\mathcal{N}}_{i,p}^+:=\{x\in\R^2\ :\ x_{i,p}+\mu_{i,p}x\in \mathcal{N}_p^+\}$.

Since for any function $v\in H^1_0(\widetilde{\mathcal{N}}_{p}^+)$ we have that the rescaled function $w(x)=v(\mip x+\xip)$ belongs to $H^1_0(\widetilde{\mathcal{N}}_{i,p}^+)$, we get that the  Dirichlet eigenvalues $\widetilde{\lambda}_j^{i,+}$, $j=1,2,\dots$ of $\widetilde{ L_{i,p}^+}$ satisfy
\[\widetilde{\lambda}_j^{i,+}=\lambda_j^+\frac{1}{p|u_p(x_{i,p})|^{p-1}},\quad j=1,2,\dots.\]
Moreover, for any subset $B\subseteq \mathcal{N}_p^+$, letting $\widetilde B_{i,p}:=\{x\in\mathbb R^2\ :\ \mu_{i,p}x+x_{i,p} \in B \}\subseteq \widetilde{\mathcal{N}}_{i,p}^+$, then the  Dirichlet eigenvalues of $\widetilde{ L_{i,p}^+}$ in
$\widetilde B_{i,p}$
are
\[
\widetilde{\lambda_j}^{i,+}(\widetilde{B}_{i,p}):=\lambda_j^+(B)\frac{1}{p|u_p(x_{i,p})|^{p-1}},\quad  j=1,2,\dots.
\]
As a consequence to prove the thesis is equivalent to show that there exists $\bar r>0$ such that
\begin{equation}\label{autovaloreRisc}
\widetilde{\lambda_1}^{i,+}\left(B_{ \bar r}(0)\right)< 0\ \mbox{ for large }p,
\end{equation}
where $B_{ \bar r}(0)$ is the ball centered in $0$ and radius $\bar r$. To prove \eqref{autovaloreRisc} we consider the functions
\[w_{i,p}:=x\cdot \nabla z_{i,p} +\frac{2}{p-1}z_{i,p}+\frac{2p}{p-1},\]
where $z_{i,p}$ is the function defined in \eqref{z_{i,p}}.
We have that $w_{i,p}$ satisfies $\widetilde{ L_{i,p}^+} (w_{i,p})=0$ and
$w_{i,p}(0)\rightarrow 2$. Moreover, as $z_{i,p}(x)\rightarrow U(x)=\log \left(\frac{1}{(1+\frac{1}{8}|x|^2)^2}\right)$,
we also get that
$w_{i,p}(x)\rightarrow -\frac{4r^2}{8+r^2}+2$,
for $|x|=r$, and so, for large $r$, $w_{i,p}(x)\rightarrow \alpha<0$ for $x\in\partial B_r(0)$.  For such $r$'s let us define $A_{i,p}:=\{x\in B_{{r}}(0): w_{i,p}>0\}$ and let us define
$\bar{w}_{i,p}=w_{i,p}$ in $A_{i,p}$ and $\bar{w}_{i,p}\equiv0$ in $B_r(0)\setminus A_{i,p}$.

Then $\bar w_{i,p}\in H^1_0(B_{r}(0))$ and for $\bar r >r$
\[\widetilde{\lambda_1}^{i,+}(B_{\bar r}(0))<\widetilde{\lambda_1}^{i,+}(B_{ r}(0))\leq \int_{B_{ r}(0)}|\nabla \bar{w}_{i,p}|^2- \int_{B_{ r}(0)}
\frac{|u_p^+(\mu_{i,p} x+ x_{i,p})|^{p-1}}{|u_p( x_{i,p})|^{p-1}}\bar{w}_{i,p}^2=0,\]
which proves the assertion.
\end{proof}

\section{Results for symmetric domains and proof of Theorem \ref{teoremaIndMorHpDeb}}\label{section:Proof}
All we have proved in the previous section holds regardless the symmetry of $\Omega$. In the sequel using the symmetry and the assumption on the Morse index \eqref{assumptionMorseIndex} we will derive more specific and precise results.

\

Thus let $\Omega\subset\R^2$ be a simply connected bounded smooth domain containing the origin and invariant under the action of a cyclic group $G$ of rotations about the origin with $|G|\geq2$. Let us consider a family $(\upp)$ of sign changing $G$-symmetric solutions as in the statement of Theorem \ref{teoremaIndMorHpDeb}. We apply Proposition \ref{prop:x1N} which gives a maximal number $k$ of families of points $(x_{i,p})$, $i=1,\ldots,k$, in $\Omega$ such that, up to a sequence, $(P^k_1)$, $(P^k_2)$ and $(P^k_3)$ hold for our solutions.
We start with the following result.

\begin{proposition}\label{MaxVaazeroMorseBasso} Let $\Omega$ be as in Theorem \ref{teoremaIndMorHpDeb} and let $(u_p)$ be a family of sign-changing $G$-symmetric solutions of \eqref{problem} satisfying \eqref{assumptionEnergyGenerale}.
Let $(x_{i,p})\subset\Omega$, $i=1,\dots ,k$ a family of points as in Proposition \ref{prop:x1N}. If $(x_{i,p})\subset \mathcal{N}_p^{+}$, for $p$ large, then assume that $m(u^{+}_p) <|G|$ otherwise, if $(x_{i,p})\subset \mathcal{N}_p^{-}$, for $p$ large, then assume that $m(u^{-}_p) <|G|$. Then
$$\frac{|x_{i,p}|}{\mu_{i,p}}\ \mbox{ is bounded}.$$
In particular $|x_{i,p}|\rightarrow 0.$
\end{proposition}
\begin{proof} We prove the assertion in the case $(x_{i,p})\subset \mathcal{N}_p^{+}$, the other case being similar. Moreover in order to simplify the notation we drop the dependence on $i$ namely we set
$x_{p}:=x_{i,p}$ and $\mu_{p}:=\mu_{i,p}$. Let $h:=|G|$ and assume by contradiction that there exists a sequence $p_n\rightarrow +\infty$ such that $\frac{|x_{p_n}|}{\mu_{p_n}}\rightarrow + \infty$.
Then, since the $h$ distinct points $g^j x_{p_n}$ (where the $(g^j)$'s are the element of $G$), $j=0,\ldots, h-1$, are the vertex of a regular polygon centered in $O$, we have that $d_n:=|g^j x_{p_n}-g^{j+1}x_{p_n}|=2\widetilde d_n \sin{\frac{\pi}{h}}$, where $\widetilde d_n:=|g^jx_{p_n}|$, $j=0,..,h-1$. Hence we also have that $\frac{d_n}{\mu_{p_n}}\rightarrow +\infty$.

Let \begin{equation}\label{R_n}R_{n}:=\min\left\{\frac{d_n}{3},\frac{d(x_{p_n},\partial\Omega)}{2},\frac{d(x_{p_n},NL_{p_n})}{2}\right\},
\end{equation}
then  by construction
\begin{equation}\label{palleDisgiunte}
\begin{array}{rl}
B_{R_n}(g^j x_{p_n})\subseteq \mathcal{N}_{p_n}^+ & \textrm{ for }j=0,\dots,h-1, \vspace{0.15cm}\\
B_{R_n}(g^j x_{p_n})\cap B_{R_n}(g^l x_{p_n}) =\emptyset, & \mbox{ for }j\neq l
\end{array}
\end{equation}
and by virtue of Corollary \ref{cor:nonvedoNL}
\begin{equation}\label{invadeR2}
\frac{R_n}{\mu_{p_n}}\rightarrow  +\infty.
\end{equation}

By Proposition \ref{LemmaFunzioneBella} it follows that $\lambda_1^+\left(B_{ \bar{r}\mu_{p_n}}(x_{p_n})\right)< 0$ for large $n$.
So by the $G$-symmetry of $u_{p_n}^+$ and the invariance of the laplacian by orthogonal transformations, it is easy to see that $\lambda_1^+(B_{\bar r\mu_{p_n}}(g^j x_{p_n}))<0$, for each $j=0,\dots, h-1$.  Hence by the variational characterization of the $1$-st eigenvalue, there exists $\varphi_j\in H^1_0(B_{\bar r\mu_{p_n}}(g^j x_{p_n}))$,  such that
$$R(v):=\frac{ \int_{\mathcal{N}_{p_n}^+}\left[|\nabla v|^2 -p|u_p^+|^{p-1}v^2 \right] }{\|v\|^2_2}  \geq R(\varphi_j)=\lambda_1^+(B_{\bar r\mu_{p_n}}(g^j x_{p_n}))< 0,$$
for any $v\in H^1_0(B_{\bar r\mu_{p_n}}(g^j x_{p_n}))$,  $v\neq 0 $,
$j=0,\dots, h-1$.

Let $W:=span\{\varphi_0, \dots, \varphi_{h-1}\}$, then by \eqref{invadeR2} it follows that for $p$ large $B_{\bar r\mu_{p_n}}(g^j x_{p_n})\subseteq B_{R_n}(g^j x_{p_n})$, hence $W\subset H^1_0 (\mathcal{N}_{p_n}^+)$ and also, by \eqref{palleDisgiunte},  $\dim W= h$  and  $R(v)\leq \sum_{j=0}^{h-1}R(\varphi_j)< 0$ for any $v\in W$.

Hence, using the variational characterization of the $h$-th eigenvalue, it follows that $\lambda_{h}^+<0$, namely $m(u^+_{p_n}) \geq h$, a contradiction.
\end{proof}

\

Now we state several results which can be obtained exactly in the same way as for analogous results in \cite{DeMarchisIanniPacella2}. They will be important steps for the proof of Theorem \ref{teoremaIndMorHpDeb}.

\begin{proposition} \label{regioneNodaleInterna} Under the same assumptions as in Theorem \ref{teoremaIndMorHpDeb} we have:
\begin{itemize}
\item[$(i)$]  $NL_p\cap\partial\Omega=\emptyset$ and $O\not\in NL_p.$
\item[$(ii)$] $O\in \mathcal{N}_p^+$ for $p$ large.
\item[$(iii)$] $ x_{i,p}\in\mathcal{N}_p^+\mbox{ for } p\mbox{ large and  }i=1,\dots,k.$
\item[$(iv)$] The maximal number  $k$ of families of points $(x_{i,p})$, $i=1,\ldots, k$, for which $(P^k_1)$, $(P^k_2)$ and  $(P^k_3)$ hold is $1$.
\item[$(v)$] There exists $C>0$ such that
for any family $(x_p)\subset \Omega$, one has
\begin{equation}\label{boundDaQ1}
\frac{|x_p|}{\mu(x_p)}\leq C
\end{equation}
for $p$ large, where $\mu(x_p)$ is defined by $(\mu(x_p))^{-2}=p|\upp(x_p)|^{p-1}$.
\end{itemize}
\end{proposition}
\begin{proof}
As already observed in the Introduction, $(i)$ is a consequence of Proposition \ref{prop:nonTocca} which applies to any $G$-symmetric solution having two nodal domains.
Once property $(i)$ is proved the $(ii)-(v)$ follow as in  \cite[Corollary 3.5, Proposition 3.6 and Corollary 3.7]{DeMarchisIanniPacella2}.
\end{proof}

\

By Lemma \ref{lemma:BoundEnergia} and Proposition \ref{regioneNodaleInterna} for the minimum points $x^-_p$ we then have
\bel\label{minimVaAZero}
 \frac{|x_p^-|}{\mu_p^-}\leq C,
\eel
so there are two possibilities: either $\fr{|x_p^-|}{\mu_p^-}\to\ell>0$ or $\fr{|x_p^-|}{\mu_p^-}\to0$ as $p\rightarrow +\infty$, up to subsequences. A crucial point of the proof is to exclude the latter case.

\begin{proposition}\label{prop:NLp+l>0}
There exists $\ell>0$ such that, up to a subsequence,
\[
\fr{|x^-_p|}{\mu_p^-}\to\ell \ \ \ \mbox{ as }\ p\rightarrow +\infty.
\]
Let us define
\begin{equation}\label{x_infty}
x_\infty:=-\lim_{p\rightarrow +\infty}\fr{x^-_p}{\mu^-_p},\ |x_\infty|=\ell>0.
\end{equation}
\end{proposition}
\begin{proof}
See \cite[Proposition 4.2]{DeMarchisIanniPacella2}.
\end{proof}
Next, even if we have no information on the geometry of the nodal line we are able to show that the nodal line shrinks to the origin faster than $\mu_p^-$ as $p\rightarrow +\infty$.

\begin{proposition}
\label{prop:NodalLineShrinks}
We have
\[
\frac{\max\limits_{y_p\in NL_p}|y_p|}{\mu_p^-}\rightarrow 0 \ \ \mbox{ as }\ p\rightarrow +\infty.
\]
\end{proposition}
\begin{proof}
See \cite[Proposition 4.3]{DeMarchisIanniPacella2}
\end{proof}
These two last propositions allow to characterize the behavior of the rescaled solutions about $x^-_p$.

\begin{proposition}
\label{prop:scalingNegativo}
The scaling of $u_p$ around $x_p^-$
\begin{equation}
\label{scalingNegativo}
v_p^-(x):=\frac{p}{u_{p}(x_{p}^-)}\left( u_{p}(\mu_{p}^- x+x_{p}^-)-u_{p}(x_{p}^-) \right)
\end{equation}
defined on $\widetilde\Omega_p^-$ converges (passing to a subsequence)  in $C^1_{loc}(\mathbb R^2\setminus \{x_{\infty}\})$ to the function
$V(x-x_{\infty})$, where $V$ is a singular solution of
\begin{equation}
\label{LiouvilleSingularEquationInZero}
\left\{
\begin{array}{lr}
-\Delta V=e^V+ H\delta_{0}\quad\mbox{ in }\R^2\\
\int_{\R^2}e^Vdx<\infty.
\end{array}
\right.
\end{equation}
for some negative $H$, and $x_{\infty}$ is the point defined in \eqref{x_infty}. More precisely, letting $\ell$ be as in \eqref{x_infty}, then:
\begin{itemize}
\item either $V$ is the radial singular solution of \eqref{LiouvilleSingularEquationInZero}, for some negative $H=H(\ell)$,
\[V=V_{rad,\ell}(x):=\log\left(\frac{2\alpha^2\beta^{\alpha}|x|^{\alpha -2}}{(\beta^{\alpha}+|x|^{\alpha})^2} \right)\qquad x\in\mathbb{R}^2\setminus\{0\},
\]
where $\alpha=\sqrt{2\ell^2+4}$ and $\beta=\ell \left(\frac{\alpha+2}{\alpha-2} \right)^{1/\alpha}$,
\item or $V$ is the $(\eta+1)$-symmetric solution of \eqref{LiouvilleSingularEquationInZero}, for $H=-4\pi\eta$, which in complex notations can be expressed as follows
$$
V=V_{\eta,\ell}(z):=\log\left(\frac{8(\eta+1)^2\lambda|z|^{2\eta}}{(1+\lambda|z^{\eta+1}-c|^2)^2}\right)\qquad z\in\mathbb{C}\setminus\{0\},
$$
where $(\eta+1)$ is an integer multiple of $|G|$, $\lambda=\frac{(\ell^2+2\eta^2)^2}{8(\eta+1)^2\ell^{2\eta+4}}$, $c=(-x_\infty)^{\eta+1}(1-\frac{4\eta(\eta+1)}{\ell^2+2\eta^2})$.
\end{itemize}
\end{proposition}
\begin{proof}
Let us consider the translations of \eqref{scalingNegativo}:
\[
s_p^-(x):=v_p^-\left(x-\frac{x_p^-}{\mu_p^-} \right)=\fr{p}{\upp(x_{p}^-)}(\upp(\mu_p^- x)-\upp(x_p^-)),\quad \quad x\in \frac{\Omega}{\mu_p^-}
\]
which solve
$$
\left\{
  \begin{array}{ll}
    -\lap s_p^-(x)=\left|1+\frac{s_p^-(x)}{p}    \right|^{p-1}\left(1+\frac{s_p^-(x)}{p}    \right) & \hbox{$x\in\frac{\Omega}{\mu^-_p}$} \\
    s^-_p(\frac{x^-_p}{\mu^-_p})=0 & \\
    s^-_p(x)\leq0 & \hbox{$x\in\frac{\Omega}{\mu^-_p}.$}
  \end{array}
\right.
$$

Observe that $\frac{\Omega}{\mu_p^-}\rightarrow \mathbb R^2$ as $p\to+\infty$.

We claim that for any fixed $r>0$, $|-\Delta s_p^-|$ is bounded in $\frac{\Omega}{\mu_p^-}\setminus B_r(0)$. \\
Indeed Proposition \ref{prop:NodalLineShrinks} implies that if $x\in \frac{\mathcal{N}^+_p}{\mu_p^-}$,  then $|x|\leq\fr{\max\limits_{z_p\in NL_p}|z_p|}{\mu^-_p}<r,$ for $p$ large, hence
\[
\left( \frac{\Omega}{\mu_p^-}\setminus B_r(0) \right)
\subset\frac{\mathcal{N}^-_p}{\mu_p^-}\ \ \ \mbox{ for } p \mbox{ large}
\]
and so the claim follows observing that for $x\in \frac{\mathcal{N}^-_p}{\mu_p^-}$, then $|-\Delta s_p^-(x)|\leq1$.\\
Hence, by the arbitrariness of $r>0$ we have that $s_p^-\rightarrow V$ in $C^1_{loc}(\mathbb R^2\setminus \{0\})$ where $V$ is a solution of
$$
-\Delta V=e^{V}\ \  \mbox{ in }\ \R^2\setminus\{0\}
$$
which satisfies $V\leq0$ and $V(-x_\infty)=0$ where $x_\infty$ is defined in \eqref{x_infty}.

\

Moreover $e^{V}\in L^1(\R^2)$, indeed for any $r>0$ and for any $\ep\in(0,1)$
\begin{eqnarray*}
\int_{B_{\fr1r}(0)\setminus B_r(0)}e^{V}\,dx&\leq&\int_{B_{\fr1r}(0)\setminus B_r(0)}\fr{|\upp(\mu^-_p x)|^{p+1}}{|\upp(x^-_p)|^{p+1}}dx+o_p(1)\\
&=&\fr{p}{|\upp(x^-_p)|^2}\int_{B_{\fr{\mu^-_p}r}(0)\setminus B_{r\mu^-_p}(0)}|\upp(y)|^{p+1}dy+o_p(1)\\
&\stackrel{Lemma\:\ref{lemma:BoundEnergia}\: (i)}{\leq}&\fr{p}{(1-\ep)^2}\int_{\Omega}|\upp(y)|^{p+1}dy+o_p(1)\stackrel{ \eqref{assumptionEnergyGenerale}}{<}+\infty.
\end{eqnarray*}
Observe that if $V$ was a classical solution of  $-\Delta V=e^V$ in the whole  $\R^2$ then necessarily $V(x)=U(x+x_{\infty})$. As a consequence $v_p^-(x)=s_p^-(x+\frac{x_p^-}{\mu_p^-})\rightarrow V(x-x_{\infty})=U(x)$ in $C^1_{loc}(\R^2\setminus\{x_{\infty}\})$. Observe that since $x_{\infty}=-\lim_p \frac{x_p^-}{\mu_{p}^-}$, then \cite[Proposition 3.8]{DeMarchisIanniPacella2} applies,  implying that $\frac{|x_p^-|}{\mu_p^-}\rightarrow 0$ as $p\rightarrow +\infty$, and this is in contradiction with Proposition \ref{prop:NLp+l>0}.

Thus, by \cite{ChenLi2,ChouWan1,ChouWan2} and the classification given in \cite{ChenLi} we have that $V$ solves, for some $\eta >0$, the following entire equation
\bel
\left\{
  \begin{array}{ll}
    -\lap V=e^{V}-4\pi\eta\delta_0 & \hbox{\textrm{in $\R^2$}} \\
    \int_{\R^2}e^{V}dx=8\pi(1+\eta), & \hbox{\:}
  \end{array}
\right.
\eel
where $\delta_0$ denotes the Dirac measure centered at the origin.

\

Since $s^-_p$ is $G$-symmetric, by the classification of \cite{PrajapatTarantello} either $V$ is radial or $\frac{\eta+1}{|G|}\in\N$ and $V$ is $(\eta+1)$-symmetric.

\

If $V$ is radial, then $V(r)$ satisfies
\[
\left\{
\begin{array}{lr}-V''-\frac{1}{r}V'=e^{V}\  \mbox{ in } (0, +\infty)\\
V\leq 0\\
V(\ell)=V'(\ell)=0
\end{array}
\right..
\]
The solutions of this problem are
\begin{equation}\label{soluzioneGenerale}
V(r)=\log\left(\frac{4}{\delta^2}\frac{e^{\frac{\sqrt{2}}{\delta}(\log r-y))}}{\left( 1+e^{\frac{\sqrt{2}}{\delta}(\log r-y))}\right)^2}   \right)-2\log r
\end{equation}
for $\delta>0, y\in\mathbb R$.

Observe that from $V'(r)=0$ we get $\frac{1-\sqrt{2}\delta}{1+\sqrt{2}\delta}=e^{\frac{\sqrt{2}}{\delta}(\log r-y)}$ and moreover $V(r)=0$ for $r=\frac{\sqrt{1-2\delta^2}}{\delta}$.
Hence by $V(\ell)=V'(\ell)=0$ it follows that $\ell^2=\frac{1-2\delta^2}{\delta^2}$ which implies that $\delta=\frac{1}{\sqrt{2+\ell^2}}$.
Inserting this estimate into \eqref{soluzioneGenerale} we get
\[
V(r)=\log\left(\frac{2\alpha^2\beta^{\alpha}r^{\alpha -2}}{(\beta^{\alpha}+r^{\alpha})^2} \right),
\]
where $\alpha=\sqrt{2\ell^2+4}$ and $\beta=\ell \left(\frac{\alpha+2}{\alpha-2} \right)^{1/\alpha}$.

\

On the other hand if $\frac{\eta+1}{|G|}\in\N$ and $V$ is $(\eta+1)$-symmetric then there exists $\lambda>0$ and $c\in \C\setminus\{0\}$ such that in complex notation
$$
V(z)=\log\left(\frac{8(\eta+1)^2\lambda|z|^{2\eta}}{(1+\lambda|z^{\eta+1}-c|^2)^2}\right),
$$
moreover $V(-x_\infty)=0$ and $V(z)\leq0$ for any $z\in\C$.

Let $\zeta\in\C$ such that $\zeta^{\eta+1}=c$ and $\zeta=\sqrt[\eta+1]{|c|}e^{i\theta}$, $\theta\in[\theta_\infty-\frac\pi{\eta+1},\theta_\infty+\frac\pi{\eta+1})$, where $-x_\infty=\ell e^{i\theta_{\infty}}$.\\
We first claim that
\begin{equation}\label{claim1}
\zeta=\sqrt[\eta+1]{|c|}e^{i\theta_\infty}.
\end{equation}
Let us suppose by contradiction that $\zeta=\sqrt[\eta+1]{|c|}e^{i\theta}$, $\theta\neq\theta_\infty$. We set $d:=\partial B_\ell(0)\cap\{t\zeta:t>0\}$. We know that $0=V(-x_\infty)\geq V(d)$ and since $|-x_\infty|^{2\eta}=|d|^{2\eta}=\ell^{2\eta}$, then $|(-x_\infty)^{\eta+1}-c|\leq|d^{\eta+1}-c|$ but this is false because $|d^{\eta+1}|=|(-x_\infty)^{\eta+1}|=\ell^{\eta+1}$ and $d^{\eta+1}=(\frac{|d|}{|\zeta|})^{\eta+1}c$. This proves \eqref{claim1}.

Next, in order to compute $\lambda$ and $c$ in terms of $x_\infty$ and $\eta$ we set:
$$
w=z e^{-i\theta_\infty}\quad\textrm{and}\quad\tilde V(w):=V(z)=\log\left(\frac{8(\eta+1)^2\lambda|w|^{2\eta}}{(1+\lambda|w^{\eta+1}-\tilde c|^2)^2}\right),
$$
where $\tilde c=e^{-i(\eta+1)\theta_\infty}c\in\R^+$.

Let us consider the restriction of the argument of the logarithm to the positive real line, namely $g(s):=\frac{8(\eta+1)^2\lambda s^{2\eta}}{(1+\lambda(s^{\eta+1}-\tilde c)^2)^2}$, $s\in(0,+\infty)$. Being $\tilde V(\ell)=V(-x_\infty)=0=\max_\C\tilde V$ we have that $g(\ell)=1$ and $g'(\ell)=0$. Imposing these two conditions we get
\begin{equation}\label{**}
8(\eta+1)^2\lambda\ell^{2\eta}=(1+\lambda(\ell^{\eta+1}-\tilde c)^2)^2,
\end{equation}
\begin{equation}\label{*}
2\eta(1+\lambda(\ell^{\eta+1}-\tilde c)^2)^2-4(\eta+1)\lambda\ell^{\eta+1}(1+\lambda(\ell^{\eta+1}-\tilde c)^2)(\ell^{\eta+1}-\tilde c)=0,
\end{equation}
and in turn combining \eqref{*} and \eqref{**} we derive
\begin{equation}\label{***}
(\ell^{\eta+1}-\tilde{c})\sqrt\lambda=\frac{\sqrt2\eta}{\ell}.
\end{equation}
Substituting \eqref{***} in \eqref{**} we get
\begin{equation}\label{****}
\lambda=\frac{(\ell+2\eta^2)^2}{8(\eta+1)^2\ell^{2\eta+4}},
\end{equation}
in turn by \eqref{***} and \eqref{****} we derive $\tilde c=\ell(1-\frac{4\eta(\eta+1)}{\ell^2+2\eta^2})$. Thus finally we have $c=(-x_\infty)^{\eta+1}(1-\frac{4\eta(\eta+1)}{\ell^2+2\eta^2})$.
\end{proof}

\

\begin{proof}[Proof of Theorem \ref{teoremaIndMorHpDeb}.]
 It follows from all previous results. More precisely, i) follows from \eqref{boundDaQ1} and Lemma \ref{lemma:BoundEnergia}. The statement ii) derives from Proposition \ref{prop:NodalLineShrinks}. The asymptotic behavior of the rescaled functions $v_p^+$ and $v_p^-$ are shown in Proposition \ref{rem:x^+=x_1} and Proposition \ref{prop:scalingNegativo}
Finally v) is a consequence of Proposition \ref{prop:x1N}.
\end{proof}

\

\end{document}